\documentclass[reqno, 11pt]{amsart}
\usepackage{amsthm, amsmath,amssymb,cite}

\usepackage{bbm}

\usepackage[margin=1in]{geometry}

\usepackage[textsize=small,backgroundcolor=orange!20]{todonotes}

\usepackage[pdfstartview=FitV, 
pagebackref=false]{hyperref}

\usepackage[color,final]{showkeys} 

\colorlet{refkey}{orange!20}
\colorlet{labelkey}{blue!30}

\newtheorem{theorem}{Theorem}[section]
\newtheorem*{theorem*}{Theorem}

\newtheorem{lemma}[theorem]{Lemma}

\newtheorem{corollary}[theorem]{Corollary}

\newtheorem*{question*}{Question}

\theoremstyle{definition}
\newtheorem{definition}[theorem]{Definition}
\newtheorem*{definition*}{Definition}

\theoremstyle{remark}
\newtheorem*{remark}{Remark}

\newcommand{\ceil}[1]{\left\lceil #1 \right\rceil}

\newcommand{\EE}{\mathbb{E}}
\newcommand{\RR}{\mathbb{R}}
\newcommand{\PP}{\mathbb{P}}
\newcommand{\NN}{\mathbb{N}}

\newcommand{\cP}{\mathcal{P}}
\newcommand{\cQ}{\mathcal{Q}}

\newcommand{\dd}{\mathsf{d}}

\DeclareMathOperator{\irreg}{irreg}
\DeclareMathOperator{\tower}{tower}

\author{Jacob Fox}
\address{Department of Mathematics, Stanford University, Stanford, CA 94305.}
\email{jacobfox@stanford.edu}
\thanks{J.~Fox is supported by a Packard Fellowship, by NSF CAREER award DMS 1352121,
and by an Alfred P. Sloan Fellowship}

\author{L\'aszl\'o Mikl\'os Lov\'asz}
\address{Department of Mathematics\\ MIT\\ Cambridge, MA 02139.}
\email{lmlovasz@math.mit.edu}

\author{Yufei Zhao}
\address{Mathematical Institute, University of Oxford, Oxford OX2 6GG, United Kingdom}
\email{yufei.zhao@maths.ox.ac.uk}
\thanks{Y.~Zhao was supported by a Microsoft Research PhD Fellowship.}

\title{On regularity lemmas and their algorithmic applications}

\date{\today}

\begin{document}

\subjclass[2010]{05C85, 05C50, 05D99}

\begin{abstract}
Szemer\'edi's regularity lemma and its variants are some of the most powerful tools in combinatorics. In this paper, we establish several results around the regularity lemma. First, we prove that whether or not we include the condition that the desired vertex partition in the regularity lemma is equitable has a minimal effect on the number of parts of the partition. Second, we use an algorithmic version of the (weak) Frieze--Kannan regularity lemma to give a substantially faster deterministic approximation algorithm for counting subgraphs in a graph. Previously, only an exponential dependence for the running time on the error parameter was known, and we improve it to a polynomial dependence. Third, we revisit the problem of finding an algorithmic regularity lemma, giving approximation algorithms for several co-NP-complete problems. We show how to use the weak Frieze--Kannan regularity lemma to approximate the regularity of a pair of vertex subsets. We also show how to quickly find, for each $\epsilon'>\epsilon$, an $\epsilon'$-regular partition with $k$ parts if there exists an  $\epsilon$-regular partition with $k$ parts. Finally, we give a simple proof of the permutation regularity lemma which improves the tower-type bound on the number of parts in the previous proofs to a single exponential bound. 

(Updated Jan 2018: Erratum added at the end. See also \cite{FLZ17n})
\end{abstract}

\maketitle

\section{Introduction}

Szemer\'edi's regularity lemma \cite{Sz76} is one of the most powerful tools in graph theory. Szemer\'edi \cite{Sz75} used an early version in the proof of his celebrated theorem on long arithmetic progressions in dense subsets of the integers. Roughly speaking, the regularity lemma says that every large graph can be partitioned into a small number of parts such that the bipartite subgraph between almost every pair of parts is random-like. 

To state Szemer\'edi's regularity lemma requires some terminology. Let $G$ be a graph, and $X$ and $Y$ be (not necessarily disjoint) vertex subsets. Let $e(X,Y)$ denote the number of pairs vertices $(x,y) \in X \times Y$ that are edges of $G$. The {\it edge density} $d(X,Y)= e(X,Y) / (|X||Y|)$ between $X$ and $Y$ is the fraction of pairs in $X \times Y$ that are edges. The pair $(X,Y)$ is {\it $\epsilon$-regular} if for all $X' \subseteq X$ and $Y' \subseteq Y$ with $|X'| \geq \epsilon|X|$ and $|Y'| \geq \epsilon |Y|$, we have $|d(X',Y')-d(X,Y)|<\epsilon$. Qualitatively, a pair of parts is $\epsilon$-regular with small $\epsilon$ if the edge densities between pairs of large subsets are all roughly the same. A vertex partition $V=V_1 \cup \ldots \cup V_k$ is {\it equitable} if the parts have size as equal as possible, that is we have $||V_i|-|V_j|| \leq 1$ for all $i,j$. An equitable vertex partition with $k$ parts is {\it $\epsilon$-regular} if all but $\epsilon k^2$ pairs of parts $(V_i,V_j)$ are $\epsilon$-regular. The regularity lemma states that for every $\epsilon>0$ there is a (least) integer $K(\epsilon)$ such that every graph has an $\epsilon$-regular equitable vertex partition into at most $K(\epsilon)$ parts. 

Arguably the main drawback of Szemer\'edi's regularity lemma is that the proof gives an enormous upper bound $K(\epsilon)$ on the number of parts, namely an exponential tower of twos of height $O(\epsilon^{-5})$. That such a huge bound is indeed necessary was an open problem for many years until Gowers \cite{Gow97} proved a lower bound on $K(\epsilon)$ which is an exponential tower of twos of height $\Omega(\epsilon^{-1/16})$. Further results by Conlon and Fox \cite{CF12} determine the dependence on the number of irregular pairs and a simpler proof of Gowers' result was obtained by Moshkovitz and Shapira \cite{MS15}. The first two authors \cite{FL} determine the tower height up to a constant factor in a version of the regularity lemma (see Section \ref{sec:equitable} for details). In this version, we show in Section \ref{sec:equitable} that the requirement that the partition is equitable has a negligible effect on the number of parts (up to changing the regularity parameter a little bit).

Due to the many applications of the regularity lemma, there has been a great deal of research on developing algorithmic versions of the regularity lemma and its applications (see the survey by Koml\'os and Simonovits \cite{KoSi}). We would like 
to be able to find an $\epsilon$-regular partition of a graph on $n$ vertices in time polynomial in $n$. Szemer\'edi's original proof of the regularity lemma was not algorithmic. The reason for this is that it needs to be able to check if a pair of parts is $\epsilon$-regular, and if not, to use subsets of the parts that realize this. This is problematic because it is shown in \cite{ADLRY} that determining whether a given pair of parts is $\epsilon$-regular is co-NP-complete. They use this to show that checking whether a given partition is $\epsilon$-regular is co-NP-complete.  

However, Alon, Duke, Lefmann, R\"odl, and Yuster \cite{ADLRY} show how to find, if a given pair of vertex subsets of size $n$ are not $\epsilon$-regular, a pair of subsets which realize that the pair is not $\epsilon^4/16$-regular. The running time is $O_\epsilon(n^{\omega+o(1)})$, where $\omega < 2.373$ is the matrix multiplication exponent (multiplying two $n \times n$ matrices in $n^{\omega+o(1)}$ time) \cite{CW90,LeGall14}. Here we use the subscript $\epsilon$ to mean that the hidden constants depend on $\epsilon$. Finding a pair of subsets of vertices that detect irregularity is the key bottleneck for the algorithmic proof of the regularity lemma. It was shown \cite{ADLRY} that one can find an $\epsilon$-regular partition with the number of parts at most an exponential tower of height $O(\epsilon^{-20})$ in an $n$-vertex graph in time $O_\epsilon(n^{\omega+o(1)})$. Thus, the following surprising fact holds: while checking whether a given partition is $\epsilon$-regular is co-NP-complete, finding an $\epsilon$-regular partition can be done in polynomial time. 

Frieze and Kannan \cite{FK99} later found a simple algorithmic proof of the regularity lemma based on a spectral approach. Using expander graphs, Kohayakawa, R\"odl, and Thoma \cite{KRT} gave a faster algorithmic regularity lemma with optimal running time of $O_\epsilon(n^2)$. Alon and Naor \cite{AN06} develop an algorithm which approximates the cut norm of a graph within a factor $0.56$ using Grothendieck's inequality and apply this to find a polynomial time algorithm which finds, for a given pair of vertex subsets of order $n$ which is not $\epsilon$-regular, a pair of subsets which realize that the pair is not $\epsilon^3/2$-regular. They further observe that their approach gives an improvement on the tower height in the algorithmic regularity lemma to $O(\epsilon^{-7})$. 

However, due to the tower-type dependence for the number of parts on the regularity parameter, these are not practical algorithms. While most graphs have a small regularity partition, the previous algorithmic proofs would not necessarily find it and would only guarantee to find a regular partition with a tower-type number of parts. Addressing this issue, Fischer, Matsliah, and Shapira \cite{EMS10} give a probabilistic algorithm which runs in constant time (depending on $\epsilon$ and $k$) which with high probability finds, in a graph which has an $\epsilon/2$-regular partition with $k$ parts, an $\epsilon$-regular partition with at most $k$ parts (implicitely defined). Tao \cite{Taoeps2} gives a probabilistic algorithm which with high probability in constant time (depending on $\epsilon$) produces an $\epsilon$-regular partition. The algorithm takes a random sample of vertices (the exact number of which is also random) and outputs the common refinement of the neighborhoods of these vertices. 

Still, it is desirable to have a fast {\it deterministic} algorithm for finding a regularity partition, which we obtain here. We give several deterministic approximation algorithms for these co-NP-complete problems.

\begin{theorem} \label{thm:partition}
There exists an $O_{\epsilon,\alpha, k}(n^2)$ time algorithm, which, given $0<\epsilon,\alpha<1$ and $k$, and a graph $G$ on $n$ vertices that admits an equitable $\epsilon$-regular partition with $k$ parts, outputs an equitable $(1+\alpha)\epsilon$-regular partition of $G$ into $k$ parts.
\end{theorem}

In other words, if a graph has a regular partition with few parts, then we can quickly find a regular partition (losing very slightly on the regularity) with the same number of parts. In particular, we obtain an algorithmic regularity lemma which is optimal in terms of the number of parts as it is exactly the same as in the non-algorithmic version (with a very slight loss on the regularity parameter).

We also give an approximation algorithm for checking whether a given pair of vertex subsets is $\epsilon$-regular, in the sense that if the pair is not $\epsilon$-regular, then we can algorithmically find a pair of vertex subsets that witness that its failure to be $(1-\alpha)\epsilon$-regular. We will formulate this in terms of regularity of bipartite graphs. We say that bipartite graph $G$ with bipartition $(X,Y)$ is $\epsilon$-regular if the pair $(X,Y)$ is $\epsilon$-regular.

\begin{theorem}\label{thm:pair}
There exists an $O_{\epsilon,\alpha}(n^2)$ time algorithm, which, given $\epsilon,\alpha>0$, and a bipartite graph $G$ between vertex sets $X$ and $Y$, each of size at most $n$, outputs one of the following:
\begin{enumerate}
\item Correctly states that $G$ is $\epsilon$-regular;
\item Finds a pair of vertex subsets $U \subseteq X$ and $W \subseteq Y$ which realize that $G$ is not $(1-\alpha)\epsilon$-regular, i.e., $|U| \ge (1 - \alpha)\epsilon |X|$, $|W| \ge (1-\alpha)\epsilon |Y|$, and $|d(U,W) - d(X,Y)| > (1-\alpha)\epsilon$.
\end{enumerate}
\end{theorem}

Using this result, by checking the regularity of each pair of parts in a partition, we have the following corollary. 

\begin{corollary}\label{cor:distinguish}
Given $\epsilon,\alpha>0$, we can distinguish in time $O_{\epsilon,\alpha}(n^2)$ between an $\epsilon$-regular partition and a partition which is not $(1-\alpha)\epsilon$-regular.   
\qed
\end{corollary}

\begin{remark}
In Theorems \ref{thm:partition} and \ref{thm:pair} and Corollary \ref{cor:distinguish}, the dependence of the running time on the parameters $\epsilon,\alpha,k$ may be improved at the cost of worsening the dependence on $n$ from $n^2$ to $n^{\omega+o(1)}$. This is because we use the recent algorithmic version of the Frieze--Kannan weak regularity lemma due to Dellamonica, Kalyanasundaram, Martin, R\"odl, and Shapira \cite{DKMRS12}, \cite{DKMRS15}. In the more recent paper \cite{DKMRS15}, they develop an $O_\epsilon(n^2)$ algorithm for finding a weak $\epsilon$-regular partition, but it has a double exponential in $1/\epsilon$ constant factor dependence. The older paper \cite{DKMRS12} has the advantage of not having this poor dependence on the regularity parameter. See Section \ref{sec:algo-weak-reg} for more information.	
\end{remark}

Counting the number of copies of a graph $H$ in another graph $G$ is a famous problem in algorithmic graph theory. For example, a special case of this problem is to determine the clique number, the size of the largest clique, in a graph. This is a well-known NP-complete problem. In fact, H\r{a}stad \cite{Has} and Zuckerman \cite{Zuck} proved that it is NP-hard  to approximate the clique number of a $n$-vertex graph within a factor $n^{1-\epsilon}$ for any $\epsilon>0$.

There is a fast {\it probabilistic} algorithm for approximating up to $\epsilon$ the fraction of $k$-tuples which make a copy of $H$. The algorithm takes $s=10\epsilon^{-2}$ samples of $k$-tuples of vertices uniformly at random from $G$ and outputs the fraction of them that make a copy of $H$. The number of copies of $H$ is a binomial random variable with standard deviation at most $s^{1/2}/2$, and hence the fraction of $k$-tuples which make a copy of $H$ in this random sample is likely within $\epsilon$ of the fraction of $k$-tuples which makes copies of $H$. However, this algorithm has no guarantee of success. It is therefore desirable to have a {\it deterministic}  algorithm for counting copies which gives an approximation for the subgraph count with complete certainty. 

The algorithmic regularity lemma is useful for deterministically approximating the number of copies of any fixed graph in a graph. Indeed, the counting lemma shows that if $k$ parts $V_1,\ldots,V_k$ are pairwise regular, then the number of copies of a graph $H$ with $k$ vertices with the copy of the $i$th vertex in $V_i$ is close to what is expected in a random graph with the same edge densities between the pairs of parts. Adding up over all $k$-tuples of parts in an $\epsilon$-regular partition and noting that almost all $k$-tuples of parts have all its pairs $\epsilon$-regular, we get an algorithm which runs in time $O_{\epsilon,k}(n^2)$ which computes the number of copies of a graph $H$ on $k$ vertices in a graph on $n$ vertices up to an additive error of $\epsilon n^k$. The major drawback with this result is the tower-type dependence on $\epsilon$ and $k$, which comes from the number of parts in the regularity lemma. 

Duke, Lefmann, and R\"odl \cite{DLR} gave a faster approximation algorithm for the number of copies of $H$ in a graph $G$. They first develop a weak regularity lemma which has an exponential dependence instead of a tower-type 
dependence. This gives an algorithm which runs in time $2^{(k/\epsilon)^{O(1)}}n^{\omega + o(1)}$ which computes the number of copies of a graph $H$ on $k$ vertices in a graph on $n$ vertices up to an additive error of $\epsilon n^k$.

In Section \ref{sec:counts}, we will use the algorithmic version of the Frieze--Kannan weak regularity lemma \cite{DKMRS12} to get the following even faster approximation algorithm for the subgraph counting problem. It improves the previous exponential dependence on the error parameter to a polynomial dependence. Here $v(H)$ and $e(H)$ denote the number of vertices and and edges in $H$, respectively.

\begin{theorem}\label{thm:countingalgthm}
Let $H$ be a graph, and let $\epsilon>0$ be given. There is a deterministic algorithm that runs in time $O_H(\epsilon^{-O(1)} n^{\omega+o(1)} + \epsilon^{- O(e(H))} n)$, and finds the number of copies of $H$ in $G$ up to an error of at most $\epsilon n^{v(H)}$.
\end{theorem}

For example, we can count the number of cliques of order $1000$ in an $n$-vertex graph up to an additive error $n^{1000-10^{-6}}$ in time $O(n^{2.4})$. 

\medskip

In the final section of the paper, we turn our attention to a regularity lemma for permutations. Cooper \cite{Coo06} proved a permutation regularity lemma which was later refined by Hoppen, Kohayakawa, and Sampaio \cite{HKS12}. In Section \ref{sec:perm-reg}, we give a new short proof of the permutation regularity lemma, improving the number of parts from tower-type to single exponential, and further extend it to an interval regularity lemma for graphs and matrices.

\section{Equitable partitions with small irregularity}
\label{sec:equitable}

Let $G$ be a graph, and $X$ and $Y$ be (not necessarily disjoint) vertex subsets. 
The \emph{irregularity} of the pair $X,Y$ is defined as
\[
\irreg(X,Y) = 
\max_{U \subseteq X, W \subseteq Y}\bigl |e(U,W)-|U||W|d(X,Y)\bigr|.
\]
The \emph{irregularity} of a partition $\cP$ of the vertex set of $G$ is defined to be
\[\irreg(\cP)=\sum_{X,Y \in \cP}\irreg(X,Y).
\] 
One version of Szemer\'edi's regularity lemma \cite[Lemma 2.2]{LS07} states that given any $\epsilon$, one can find an $M(\epsilon)$ such that any graph $G$ has a partition into at most $M(\epsilon)$ parts with irregularity at most $\epsilon |V|^2$. The proof of the regularity lemma gives $M(\epsilon) \le \tower(O(\epsilon^{-2}))$, and we now know that this is essentially tight \cite{Gow97},\cite{CF12},\cite{FL}, in the sense that $M(\epsilon) = \tower(\Theta(\epsilon^{-2}))$ \cite{FL}. Here the tower function is defined by $\tower(1)=2$ and $\tower(k+1)=2^{\tower(k)}$. 

We say that a partition is \emph{equitable} if any two parts differ in size by at most one. It is a convenient property to have in a regularity partition. The main result of this section shows that for any vertex partition, one can refine it a bit further to obtain a partition which is close to an equitable partition whose irregularity is not substantially larger.

\begin{theorem} \label{thm:mainequi} 
	Let $0 < \alpha < 1/2$, and $m$ be a positive integer, and let $G$ be a graph on $n \ge 10^8 m\alpha^{-5}$ vertices.
	If $\cP$ is a vertex partition of $G$ into $m$ parts, then there is an equitable vertex partition $\cQ$ of $G$ into at most $4m/\alpha$ parts such that $\irreg(\cQ) \leq  \irreg(\cP) + \alpha n^2$. 
\end{theorem}

Let $M_{eq}(\epsilon)$ be the smallest $M$ such that, for any graph $G = (V,E)$, there is an equitable partition into at most $M$ parts with total irregularity is at most $\epsilon |V|^2$. We have $M_{eq}(\epsilon) \geq M(\epsilon)$ trivially. As a consequence of Theorem \ref{thm:mainequi}, we show directly that adding the condition that the partition is equitable has a very small effect on the size of the smallest partition with small irregularity. 

\begin{theorem} \label{thm:equinotmuchmore}
Let $0<\epsilon <1$ and  $0<\alpha < 1/2$. We have $M_{eq}(\epsilon+\alpha) \leq  \alpha^{-O(1)}M(\epsilon)$. 
\end{theorem}

In particular, taking $\alpha$ small but not too small, such as $\alpha=2^{-1/\epsilon}$, we see that the tower height in Szemer\'edi's regularity lemma is not significantly affected by adding the equitability requirement. 

Note that Theorem~\ref{thm:mainequi} also applies to graphs $G$ whose number of parts in the regularity partition is not  as large as the worst case $M(\epsilon)$. To prove Theorem \ref{thm:mainequi}, we randomly divide each part of the partition $\cP$ into parts of (essentially) equal size (apart from a small remaining subset), and then arbitrarily partition the relatively few remaining vertices into parts of equal size to obtain an equitable partition. We show that this works with high probability.

As a first step, the following lemma shows that with high probability, a pair of random subsets $X',Y'$ of a pair of parts $X,Y$ induces roughly the same subgraph density as $X$ and $Y$.
 
\begin{lemma} 
\label{lem:random-subsets-preserve-density}
Let $X$ and $Y$ be vertex subsets of a graph $G$. Let $X' \subseteq X$ and $Y' \subseteq Y$ be picked uniformly at random with $|X'| = |Y'|=k$. Then
\[
\PP\left( | d(X',Y')-d(X,Y)| < \delta \right) \ge 1 - 2e^{-\delta^2k/4}.
\]
\end{lemma}

\begin{proof}
Consider picking the vertices of $X'$ and $Y'$ one at the time, starting with the ones in $X'$. Let $Z_0,\ldots,Z_{2k}$ be the martingale where $Z_i$ is the expected value of $e(X',Y')$ conditioned on knowing the first $i$ vertices already chosen (this is sometimes referred to as the vertex-exposure martingale). We have $|Z_{i}-Z_{i-1}| \leq k$ as the choice of each vertex in $X'$ and $Y'$ changes the final $e(X',Y')$ by at most $k$. 
By the Azuma--Hoeffding inequality (see Chapter 7 of Alon and Spencer \cite{AlSp}),
\[
\PP\left( |Z_{2k}-Z_0| \geq t \right) \leq 2 e^{-t^2/(4k^3)}.
\]
We have $Z_0= k^2 d(X,Y)$ and $Z_{2k}=e(X',Y')$. 
Set $t=\delta k^2$, we obtain
\[
\PP\left( | e(X',Y')- k^2 d(X,Y) | \geq \delta k^2 \right) \leq 2e^{-\delta^2 k / 4}.
\]
Noting that $e(X',Y') = k^2d(X',Y')$, the lemma follows.
\end{proof}

The next lemma show that the irregularity parameter remains roughly the same when restricted to a random, much smaller, subset of vertices. Recall that the cut metric $\dd_\square$ between two graphs $G$ and $H$ on the same vertex set $V = V(G) = V(H)$ is defined by
\[
\dd_\square(G,H) := \max_{U,W \subseteq V} \frac{|e_G(U,W) - e_H(U,W)|}{|V|^2}.
\]
When $G$ and $H$ are bipartite graphs on $V = X \cup Y$, we define the cut metric as
\[
\dd_\square(G,H) := \max_{U \subseteq X, W \subseteq Y} \frac{|e_G(U,W) - e_H(U,W)|}{|X||Y|}.
\]
We also use the same notation for edge-weighted graphs, where $e(U,W)$ denotes the sum of weights of all edges in $U \times W$.

\begin{lemma}\label{lem:random-subsets-preserve-irreg}
Let $X,Y$ be vertex subsets of a graph $G$. Let $X' \subseteq X$ and $Y' \subseteq Y$ be picked uniformly at random with $|X'|=|Y'|=k$. Then with probability at least $1 - 6e^{-\sqrt{k}/10}$,
\begin{equation}
	\label{eq:irreg-diff}
	\left| \frac{\irreg(X',Y')}{k^2}-\frac{\irreg(X,Y)}{|X||Y|} \right| \le \frac{9}{k^{1/4}},
\end{equation}
\end{lemma}

\begin{proof}
We use the so-called First Sampling Lemma \cite[Theorem 2.10]{BCLSV08} (we quote the statement from \cite[Lemma 10.5]{Lov}): if $G$ and $H$ are weighted graphs with $V(G) = V(H)$ and edge weights in $[0,1]$, and $S \subseteq V(G)$ is chosen uniformly at random with $|S| = k$, then with probability at least $1 - 4 e^{-\sqrt{k}/10}$, 
\[
\left| \dd_\square(G[S], H[S]) - \dd_\square(G, H) \right| \le \frac{8}{k^{1/4}}.
\]
Let $G[X,Y]$ denote the bipartite (weighted) graph with vertex sets $X$ and $Y$, and whose edges are induced from $G$. A bipartite version of this sampling lemma holds true, that with probability at least $1 - 4 e^{-\sqrt{k}/10}$,
\begin{equation} 
\label{eq:sampling-lemma}
\left| \dd_\square(G[X',Y'], H[X',Y']) - \dd_\square(G[X,Y], H[X,Y]) \right| \le \frac{8}{k^{1/4}},
\end{equation}
and its proof is nearly identical to the first version stated above. Note that 
\[
\irreg(X, Y) = |X||Y| \dd_\square(G[X,Y], d(X, Y)),
\]
where the second argument denotes the complete graph with loops with all edge weights equal to $d(X,Y)$. Similarly, 
\[
\irreg(X',Y') = k^2 \dd_\square(G[X',Y'], d(X', Y')).
\]
By letting $H$ in \eqref{eq:sampling-lemma} be the complete graph with loops and all edge weights $d(X, Y)$, we obtain that with probability at least $1 - 4 e^{-\sqrt{k}/10}$,
\[
\left| \frac{\irreg(X',Y')}{k^2}-\frac{\irreg(X,Y)}{|X||Y|} \right|  \le \frac{8}{k^{1/4}} + |d(X, Y) - d(X', Y')|.
\]
We then apply Lemma~\ref{lem:random-subsets-preserve-density} with $\delta = 1/k^{1/4}$ to reach the desired conclusion.
\end{proof}

As a corollary of Lemma~\ref{lem:random-subsets-preserve-density}, and noting that the LHS of \eqref{eq:irreg-diff} is at most 1, we have
\begin{equation} \label{eq:irreg-subset-upper-bound}
	\EE \left( \frac{\irreg(X',Y')}{k^2} \right)
	\le \frac{\irreg(X,Y)}{|X||Y|} + \frac{9}{k^{1/4}} + 6e^{-\sqrt{k}/10}
	\le \frac{\irreg(X,Y)}{|X||Y|} + \frac{20}{k^{1/4}}.
\end{equation}

\begin{proof}[Proof of Theorem \ref{thm:mainequi}]
We shall omit floors and ceilings for the sake of clarity of presentation. Let $k = \alpha n/(4m)$. Let $V_1,\ldots,V_m$ be the parts of $\cP$. Uniformly at random partition each $V_i$ into parts of size $k$, with possibly one remainder part of size less than $k$. Call $\cP'$ be the resulting partition.  

By \eqref{eq:irreg-subset-upper-bound},\footnote{It is easy to modify the proof to address the case when $X'$ and $Y'$ are within the same part of $\mathcal{P}$.}
\begin{align*}
	\EE\Biggl(\sum_{\substack{X',Y' \in \cP' \\ |X'| = |Y'| = k}} \irreg(X',Y')\Biggr)
	&\le 
	\sum_{X,Y \in \cP} \left(\irreg(X,Y) + \frac{20}{k^{1/4}}|X||Y| \right)
	\\
	&\le \irreg(\cP) + \frac{20n^2}{k^{1/4}}.
\end{align*}
So there exists some such partition $\cP'$ such that 
\begin{equation} \label{eq:irreg-sum}
	\sum_{\substack{X',Y' \in \cP' \\ |X'| = |Y'| = k}} \irreg(X',Y')
	\le \irreg(\cP) + \frac{20n^2}{k^{1/4}}.
\end{equation}
Fix $\cP'$ to be this partition.

Let $S$ be the union of the parts of $\cP'$ of size less than $k$. 
There is at most one such part of $\cP'$ for each $V_i$, so $|S| < mk$. 
Arbitrarily partition $S$ into sets of size $k$, and let $\cQ$ be the equitable vertex partition consisting of these parts of $S$ along with the parts of $\cP'$ of size $k$.

Parts arising from $S$ contributes at most $2|S|n < 2mkn$ to $\irreg(\cQ)$, whereas the other contributions to $\irreg(\cQ)$ are bounded by \eqref{eq:irreg-sum}. Thus
\[
\irreg(\cQ) \le \irreg(\cP) + \frac{20n^2}{k^{1/4}} + 2mkn
\le \irreg(\cP) + \alpha n^2,
\]
where the last step follows from
\[
\frac{20}{k^{1/4}} = \frac{20(4m)^{1/4}}{(\alpha n)^{1/4}} \le \frac{20(4m)^{1/4}}{(\alpha \cdot 10^{8} m\alpha^{-5})^{1/4}} < \frac{\alpha}{2}
\]
and 
\[
\frac{2mk}{n} = \frac{2m}{n}\frac{\alpha n}{4m} \le \frac{\alpha}{2}.
\]
Therefore $\cQ$ is the required equipartition.
\end{proof}

As a consequence of Theorem \ref{thm:mainequi}, we can prove Theorem \ref{thm:equinotmuchmore}.

\begin{proof}[Proof of Theorem \ref{thm:equinotmuchmore}]
Let $G$ be a graph with a partition $\cP$ of the vertex set into $m \le  M(\epsilon)$ parts with irregularity at most $\epsilon n^2$. If $n < 10^{8} m \alpha^{-5}$, then we just partition the vertices into singleton sets, which has zero irregualarity, using at most $10^{8} m \alpha^{-5} \le 10^{8} \alpha^{-5} M(\epsilon)$ parts. 
Otherwise apply Theorem~\ref{thm:mainequi}  to obtain a partition with irregularity at most $(\epsilon+\alpha)n^2$, and at most $4m/\alpha \le 4\alpha^{-1} M(\epsilon)$ parts.
\end{proof}

\section{Algorithmic weak regularity}
\label{sec:algo-weak-reg}

In this section, we review some results concerning algorithmic versions of the Frieze--Kannan weak regularity lemma. We will be applying these results in subsequent sections.

Given any edge-weighted graph $G$ and any partition $\cP \colon V = V_1 \cup V_2 \cup \dots \cup V_t$ of the vertex set of $G$ into $t$ parts, let $G_\cP$ denote the weighted graph with vertex set $V$ obtained by giving weight $d_{ij} := d(V_i,V_j) = e(V_i,V_j)/(|V_i||V_j|)$ to all pairs of vertices in $V_i \times V_j$, for every $1 \le i \le j \le t$. We say $\cP$ is an \emph{$\epsilon$-regular Frieze--Kannan} (or \emph{$\epsilon$-FK-regular}) partition if $\dd_\square(G,G_\cP)\le \epsilon$. In other words, $\cP$ is an $\epsilon$-regular Frieze--Kannan partition if
\begin{equation} \label{eq:FK-pair}
\left\lvert e(S,T) - \sum_{i,j=1}^t d_{ij} |S \cap V_i||T \cap V_j| \right\rvert \le \epsilon |V|^2.
\end{equation}
for all $S,T \subseteq V$. We say that sets $S$ and $T$ \emph{witness} that $\cP$ is not $\epsilon$-FK-regular if the above inequality is violated.

Frieze and Kannan~\cite{FK99a} proved the following regularity lemma. 

\begin{theorem}[Frieze--Kannan]
Let $\epsilon > 0$. Every graph has an $\epsilon$-regular Frieze--Kannan partition with at most $2^{2/\epsilon^2}$ parts.
\qed
\end{theorem}

We apply efficient deterministic algorithms for generating Frieze--Kannan regular partitions. Such algorithms were recently given in \cite{DKMRS12,DKMRS15} by Dellamonica et al. Specifically, in \cite{DKMRS12}, the authors gave an $\epsilon^{-6} n^{\omega+o(1)}$ time algorithm to generate an equitable $\epsilon$-regular Frieze--Kannan partition of a graph on $n$ vertices into at most $2^{O(\epsilon^{-7})}$ parts. Recall that $\omega < 2.373$ is the matrix multiplication exponent. In \cite{DKMRS15} a different algorithm was given which improved the dependence of the running time on $n$ from $O_\epsilon(n^{\omega+o(1)})$ to $O_\epsilon(n^2)$, while sacrificing the dependence of $\epsilon$. Namely, it was shown that there is a deterministic algorithm that finds, in $O(2^{2^{\epsilon^{-O(1)}}} n^2)$ time, an $\epsilon$-regular Frieze--Kannan partition into at most $2^{\epsilon^{-O(1)}}$ parts. It remains an open problem to improve the dependence on $\epsilon$ in the running time.

The proof of the Frieze--Kannan regularity lemma and its algorithmic versions, roughly speaking, run as follows:
\begin{itemize}
	\item Given a partition (starting with the trivial partition with one part), either it is $\epsilon$-FK-regular (in which case we are done), or we can exhibit some pair of subsets $S,T$ of vertices that witness the irregularity by violating \eqref{eq:FK-pair} (in the algorithmic versions, one may only be guaranteed to find $S$ and $T$ that violate \eqref{eq:FK-pair} for some smaller value of $\epsilon$).
	\item Refine the partition by using $S$ and $T$ to split each part into at most four parts, thereby increasing the total number of parts by a small factor. 
	\item Repeat. Use a mean-square-density increment argument to upper bound the number of possible iterations.
\end{itemize}

\begin{remark}
As in the case of the usual regularity lemma, it is possible to obtain an equitable partition in the Frieze--Kannan regularity lemma, increasing the number of parts by a factor of at most 4. We will not need this for our algorithm, however.
\end{remark}

For example, in the algorithmic version \cite{DKMRS12}, the first step (also the key step) is given as the following result \cite[Corollary 3.1]{DKMRS12}.

\begin{theorem}\label{thm:DKMRS12-pair}
There is an $n^{\omega+o(1)}$-time algorithm which, given $\epsilon > 0$, an $n$-vertex graph $G$ and a partition $\cP$ of $V (G)$, does one of the following:
\begin{enumerate}
\item Correctly states that $\cP$ is $\epsilon$-FK-regular;
\item Finds sets $S$, $T$ which witness the fact that $\cP$ is not $\epsilon^3 /1000$-FK-regular. \qed
\end{enumerate} 
\end{theorem}

In \cite{DKMRS15}, an alternative theorem is given for finding an irregular pair. The statement below is a consequence of \cite[Theorem 8.1]{DKMRS15}.

\begin{theorem}\label{thm:DKMRS15-pair}
There is an $O(2^k n^2)$-time algorithm which, given $\epsilon > 0$, an $n$-vertex graph $G$ and a partition $\cP$ of $V (G)$ into $k$ parts, does one of the following:
\begin{enumerate}
\item Correctly states that $\cP$ is $\epsilon$-FK-regular. 
\item Finds sets $S$, $T$ which witness the fact that $\cP$ is not $\epsilon^{O(1)}$-FK-regular.  \qed
\end{enumerate} 
\end{theorem}

\medskip

There is a variant of the weak regularity lemma, where the final output is not a partition of $V$ into $2^{\epsilon^{-O(1)}}$ parts, but rather an approximation of the graphs as a sum of $\epsilon^{-O(1)}$ complete bipartite graphs each assigned some weight. Below by \emph{weighted graph} we mean a graph with edge-weights. For $S,T \subseteq V$, by $K_{S,T}$ we mean the weighted graph where an edge $\{s,t\}$ has weight 1 if $s \in S$ and $t \in T$ (and weight 2 if $s,t \in S \cap T$) and weight zero otherwise. For any $c \in \RR$, by $c G$ we mean the weighted graph obtained from $G$ by multiplying every edge-weight by $c$. For a pair of weighted graphs $G_1,G_2$ on the same set of vertices, we will use the notation $G_1+G_2$ to denote the graph on the same vertex set with edge weights summed (and weight $0$ corresponding to not having an edge). Additionally, we write $c$ to mean the constant graph with all edge-weights equal to $c$. 

\begin{theorem}[Frieze--Kannan] \label{thm:FK-overlay}
Let $\epsilon > 0$. Let $G$ be any weighted graph with $[-1,1]$-valued edge weights. There exists some $k \le O(\epsilon^{-2})$, subsets $S_1, \dots, S_k, T_1, \dots, T_k \subseteq V$, and $c_1, \dots, c_k \in [-1,1]$, so that  
\[
\dd_\square(G, d(G) + c_1K_{S_1, T_1} + \dots + c_k K_{S_k,T_k}) \le \epsilon.
\]
\qed
\end{theorem}

See \cite[Lemma 4.1]{LS07} for a proof (given there in a more general setting of arbitrary Hilbert spaces). Roughly speaking, to find the appropriate $S_i, T_i, c_i$, in the second step of the above outline of the proof of the weak regularity lemma, instead of using $S$ and $T$ to refine the existing partition, we subtract $c K_{S,T}$ from the remaining weighted graph (starting with $G$), where $c$ is the density between $S$ and $T$ in the remaining weighted graph. We record the corresponding $S_i, T_i, c_i$ in step $i$ of this iteration. We can bound the number of iterations by observing that the $L^2$ norm of $G - d(G) - c_1K_{S_1,T_1} - \dots - c_iK_{S_i,T_i}$ must decrease by a certain amount at each step.

As for the algorithmic version, using Theorems~\ref{thm:DKMRS12-pair} and \ref{thm:DKMRS15-pair} (or minor modifications thereof), we can efficiently approximate $G$ as a weighted sum of $\epsilon^{-O(1)}$ complete bipartite graphs.

\begin{corollary} \label{cor:reg-sum}
There exists a
$\min\{ \epsilon^{-O(1)}n^{\omega+o(1)}, O(2^{2^{\epsilon^{-O(1)}}}n^2) \}$ 
time algorithm which, given $\epsilon > 0$ and an $n$-vertex graph $G$, outputs subsets $S_1, \dots, S_k, T_1, \dots, T_k \subseteq V(G)$ and real numbers $c_1, \dots, c_k$, for some $k \le \epsilon^{-O(1)}$, such that
\[
\dd_\square(G, d(G) + c_1K_{S_1, T_1} + \dots + c_k K_{S_k,T_k}) \le \epsilon.
\]
\qed
\end{corollary}

We will use a variation for bipartite graphs.

\begin{corollary} \label{cor:reg-sum-bip}
There exists a
$\min\{ \epsilon^{-O(1)}n^{\omega+o(1)}, O(2^{2^{\epsilon^{-O(1)}}}n^2) \}$ 
time algorithm which, given $\epsilon > 0$ and a bipartite graph $G$ between vertex sets $X$ and $Y$ each with at most $n$ vertices, outputs subsets $S_1, \dots, S_k \subseteq X$ and $T_1, \dots, T_k \subseteq Y$ and real numbers $c_1, \dots, c_k$, for some $k \le \epsilon^{-O(1)}$, such that
\[
\dd_\square(G, d(X,Y) + c_1K_{S_1, T_1} + \dots + c_k K_{S_k,T_k}) \le \epsilon,
\]
where the constant $d(X,Y)$ here denotes the weighted complete bipartite graph $d(X,Y) K_{X,Y}$.
\qed
\end{corollary}

\section{Approximation algorithm for subgraph counts}
\label{sec:counts}

Suppose that we are given a graph $G$ on $n$ vertices, and we would like to find the number of copies of a small graph $H$ on $k$ vertices that are contained in $G$. We would like to count them up to an error at most $\epsilon n^k$. In this section, we will provide a deterministic algorithm that can do so. Specifically, we prove Theorem~\ref{thm:countingalgthm}, reproduced below for convenience.

\begin{theorem*}
Let $H$ be a graph, and let $\epsilon>0$ be given. There is a deterministic algorithm that runs in time $O_H(\epsilon^{-O(1)} n^{\omega+o(1)} + \epsilon^{- O(e(H))} n)$, that finds the number of copies of $H$ in $G$ up to an error of at most $\epsilon n^{v(H)}$.
\end{theorem*}

It will be cleaner to work instead with $\hom(H, G)$, the number of graph homomorphisms from $H$ to $G$. This quantity differs from the number of (labeled) copies of $H$ in $G$ by a negligible $O(n^{v(H)-1})$ additive error (the hidden constants here and onward may depend on $H$). We extend the definition of $\hom(H,G)$ to edge-weighted graphs $G$: if the edge $xy$ in $G$ has weight $G(x,y)$, then we define
\[
\hom(H,G)=\sum_{f:V(H) \rightarrow V(G)} \prod_{\{u,v\} \in E(H)}G(f(u),f(v)).
\] 
Note that here $G(x,y)$ is defined on all pairs, with $G(x,y)=0$ if there is no edge between $x$ and $y$.

The idea is to apply a weak regularity lemma in the form of Corollary~\ref{cor:reg-sum}. A weakly regular approximation also gives an approximation of $H$-count, via a standard counting lemma (see \cite[Lemma 10.22]{Lov}):

\begin{lemma}[Counting lemma]\label{lem:counting}
Given any graph $H$ and any two weighted graphs $G_1$ and $G_2$ on the same set $V$ of $n$ vertices, we have
\[
|\hom(H, G_1) - \hom(H, G_2)| \le e(H) \dd_\square(G_1, G_2) n^{v(H)}.
\]
\qed
\end{lemma}

Here is the algorithm. Apply Corollary~\ref{cor:reg-sum} to find any approximation 
\[
G' = d(G) + c_1K_{S_1, T_1} +\dots + c_k K_{S_k,T_k}
\]
of $G$ with $\dd_\square(G,G') \le \epsilon/e(H)$ and $k \le \epsilon^{-O(1)}$. By the counting lemma, it suffices to compute
\begin{equation} \label{eq:hom-expand}
\hom(H,G') = \hom(H, d(G) + c_1K_{S_1, T_1} +\dots + c_k K_{S_k,T_k}),
\end{equation}
which can be done in $O_H(k^{e(H)}n)$ time, as follows. We can expand the right-hand side of \eqref{eq:hom-expand} via the distributive property, writing
\begin{equation} \label{eq:hom-expand-distr}
\hom(H,G') = \sum_{\varphi \colon E(H) \to \{0, \dots, k\}} \hom^\varphi(H, (d(G), c_1K_{S_1,T_1},\dots, c_kK_{S_k,T_k}))
\end{equation}
where for each assignment $\varphi \colon E(H) \to \{0, \dots, k\}$ of edges of $H$ to the components of $G'$ we write $\hom^\varphi(H, (G_0, G_1, \dots, G_k))$ to mean the homomorphism count obtained where the image of each edge $e \in E(H)$ is restricted to $G_{\varphi(e)}$, i.e.,
\[
\hom^\varphi(H, (G_0, G_1, \dots, G_k)) 
= \sum_{f \colon V(H) \to [n]} \prod_{uv \in E(H)} G_{\varphi(uv)}(f(u),f(v)).
\]
Here by $G_i(x,y)$ we mean the edge-weight of $(x,y)$ in $G_i$.

There are $(k+1)^{e(H)}$ possible maps $\varphi$. We claim that each term on the right-hand side of \eqref{eq:hom-expand-distr}, corresponding to some $\varphi$, can be exactly computed in $O_H(n)$ time. Taking out constant factor, it remains to compute the value of $\hom^\varphi(H, (1, K_{S_1, T_1},\dots, K_{S_k,T_k}))$. We further decompose each $K_{S_i,T_i}$ (viewed as an adjacency matrix) as a sum $1_{S_i \times T_i} + 1_{T_i \times S_i}$ and apply the distributive property once again to expand the quantity as a sum of $2^{e(H)}$ terms. Each term counts the number of maps $f \colon V(G) \to [n]$ such that, for every $v \in V(H)$, $f(v) \in \bigcap_{e \in E(H)} R_{\varphi(e)}$ for some choice of $R_{\varphi(e)} = S_{\varphi(e)}$ or $T_{\phi(e)}$. The size of such an intersection can be computed in $O_H(n)$ time, and we can compute this term (one of $2^{e(H)}$ terms) by multiplying over all $v \in V(H)$.
There are $2^{e(H)}$ choices for which summand in $1_{S_i \times T_i} + 1_{T_i \times S_i}$ to take in the expansion over all $i$, and by summing over all $2^{e(H)}$ choices, we can evaluate $\hom^\varphi(H, (1, K_{S_1, T_1},\dots, K_{S_k,T_k}))$. By summing over all $\varphi$ in~\eqref{eq:hom-expand-distr}, we see that $\hom(H,G')$ can be exactly computed in $O_H(k^{e(H)} n)$ time, thereby providing the desired approximation to $\hom(H,G)$.

As for the running time, it took $\epsilon^{-O(1)}n^{\omega+o(1)}$ time to find the approximation $G'$, and it took $O_H(k^{e(H)}n) = O_H(\epsilon^{-O(e(H))}n)$ time to compute $\hom(H, G')$, giving the claimed total running time.

\section{Finding an irregular pair}

In this section, we prove Theorem~\ref{thm:pair}, reproduced below for convenience.

\begin{theorem*}
There exists an $O_{\epsilon,\alpha}(n^2)$ time algorithm, which, given $\epsilon,\alpha>0$, and a bipartite graph $G$ between vertex sets $X$ and $Y$, each of size at most $n$, outputs one of the following:
\begin{enumerate}
\item Correctly states that $G$ is $\epsilon$-regular;
\item Finds a pair of vertex subsets $U \subseteq X$ and $W \subseteq Y$ which realize that $G$ is not $(1-\alpha)\epsilon$-regular, i.e., $|U| \ge (1 - \alpha)\epsilon |X|$, $|W| \ge (1-\alpha)\epsilon |Y|$, and $|d(U,W) - d(X,Y)| > (1-\alpha)\epsilon$.
\end{enumerate}
\end{theorem*}

We can assume that $\alpha<1/2$ since for larger $\alpha$ we can just apply the algorithm with a lower value of $\alpha$. We shall give an $O(2^{2^{(\alpha\epsilon)^{-O(1)}}} n^2)$-time algorithm. Using Corollary~\ref{cor:reg-sum-bip}, we approximate $G$ by $G' = d(G) + c_1K_{S_1,T_1} + \dots c_k K_{S_k,T_k}$ so that $k = (\alpha\epsilon)^{-O(1)}$ and $\dd_\square(G,G') \le \alpha\epsilon^3/4$. Here $S_1, \dots, S_k \subseteq X$ and $T_1, \dots, T_k \subseteq Y$.  We shall assume that $k$ is small compared to $|X|$ and $|Y|$, namely,
\begin{equation} \label{eq:k-small}
	100 \cdot k2^k \le \alpha \epsilon^3 \min\{|X|,|Y|\},
\end{equation}
for otherwise  we can accomplish the task by a complete search (say when $|X| \le |Y|$) over all subsets of $X$ in $2^{O(|X|)} = 2^{O(\alpha^{-1}\epsilon^{-3}k2^k)} = 2^{2^{(\alpha\epsilon)^{-O(1)}}}$ time, which is enough.

We say that a sequence of numbers $u, u_1, \dots, u_k, w, w_1, \dots, w_k$ is \emph{feasible} if there exists a function $\mu \colon \colon X \cup Y \to [0,1]$ (we write $\mu(S) = \sum_{x \in S} \mu(x)$ from now on) such that  the following quantities
\[
\frac{|\mu(X) - u|}{|X|}, \frac{|\mu(Y) - w|}{|Y|},
 \frac{|\mu(S_i) - u_i|}{|X|}, \frac{|\mu(T_i) - t_i|}{|Y|}, \text{ for all } 1 \le i \le k,
\]
are each at most $\alpha \epsilon^3 / (100 k)$. One can think of $\mu$ as representing subsets $U \subseteq X$ and $W \subseteq Y$ with $[0,1]$-valued weights attached to its elements. One can determine via a linear program if a given sequence is feasible (see Lemma~\ref{lem:feasible-LP} below).

Here is the algorithm. We perform a complete search through all sequences $u, u_1, \dots, u_k, w, w_1, \dots, w_k$ of nonnegative integers at most $n$, where $u$ and each $u_i$ are divisible by $\lfloor \alpha \epsilon^3 |X| / (100 k) \rfloor$, and $w$ and each $w_i$ are divisible by $\lfloor \alpha \epsilon^3 |Y| / (100 k) \rfloor$. For each such sequence, we check if it is feasible, and if so then we check whether the inequalities
\begin{equation} \label{eq:surrogate-reg}
\left| \sum_{i=1}^k c_i u_i w_i \right| > (1 - \alpha/2)\epsilon u w, \quad
u \ge (1-\alpha/2)\epsilon |X|, \quad  \text{and }
w \ge (1-\alpha/2)\epsilon |Y|	
\end{equation}
hold. If they never hold for any feasible sequence, then we state that $G$ is $\epsilon$-regular. On the other hand, if they hold for some feasible sequence, then we can convert $f$ into actual sets $U$ and $W$ (as we shall explain) that witness that $G$ is not $(1-\alpha)\epsilon$-regular.

Next we prove the correctness of the algorithm if the output is that $G$ is $\epsilon$-regular.

Consider the partition of $X$ given by the common refinement by $S_1, \dots, S_k$. For any index set $I \subseteq [k]$, let $S_I = (\bigcap_{i \in I} S_i) \cap (\bigcap_{i \notin I} (X \setminus S_i))$ denote the part in the common refinement indexed by $I$. We can compute the sizes $|S_I|$ for all $I \subseteq [k]$ in $O(2^k n)$ time. With this information at hand:

\begin{lemma} \label{lem:feasible-LP}
There exists a $2^{O(k)}$ time algorithm that determines whether a given sequence $u$, $u_1, \dots, u_k$, $w$, $w_1, \dots, w_k$ is feasible.	
\end{lemma}

\begin{proof}
It suffices to show that one can determine in the required time whether there exists $\mu \colon X \to [0,1]$ such that $|\mu(X) -u| \le a$ and $|\mu(S_i) - u_i| \le a_i$, for each $i$. Here $a = a_i = \lfloor \alpha\epsilon^3n/(100k)\rfloor$ is the required bound (though it could be chosen arbitrarily for the purpose of this lemma). The situation for $Y$ is analogous.

For the purpose of satisfying the inequalities $|\mu(X) -u| \le a$ and $|\mu(S_i) - u_i| \le a_i$, one only needs to know the sum of values of $\mu$ on parts in the partition of $X$ induced by the common refinement of $S_1, \dots, S_k$.

For each $I \subseteq [k]$, the variable $x_I$ is supposed to correspond to the value of $\mu(S_I)$. Then $\mu$ exists if and only if there exists $(x_I)_{I \subseteq [k]} \in \RR^{2^k}$ satisfying the following inequalities:
\begin{align*}
	-a \le \left(\sum_{I \subseteq [k]} x_I\right) - u \le a, & \\
	-a_i \le \left(\sum_{I \ni i} x_I \right) - u_i \le a_i &\qquad \text{ for all } i \in [k], \\
	\text{and}\qquad 
	0 \le x_I \le |S_I| &\qquad \text{ for all } I \subseteq [k].
\end{align*}
This is a linear program in $2^k + 1$ variables, which can be solved in $2^{O(k)}$ time. The original sequence is feasible if and only if the above system of linear inequalities has some solution in $(x_I)$.
\end{proof}

Suppose the algorithm does not find any feasible sequence satisfying \eqref{eq:surrogate-reg}. We claim that $G$ is $\epsilon$-regular. Assume otherwise. Then there exist $U \subseteq X$ and $W \subseteq Y$ such that $|U| \ge \epsilon|X|$, $|W| \ge \epsilon|Y|$, and $|d(U,W) - d(X,Y)| > \epsilon$. Since 
$\dd_\square(G, G') \le \alpha\epsilon^3/4$, we have 
$|e_G(U,W) - e_{G'}(U,W)| \le (\alpha \epsilon^3/4) |X||Y|$. Thus
\begin{align*}
|e_{G'}(U,W) - d_G(X,Y)|U||W||
&\ge |e_G(U,W) - d_G(X,Y)|U||W|| - |e_G(U,W) - e_{G'}(U,W)| \\
&\ge |d_G(U,W) - d_G(X,Y)||U||W| - \tfrac14 \alpha \epsilon^3 |X||Y| \\
&\ge \epsilon |U||W| - \tfrac14 \alpha \epsilon |U||W| \\
&\ge (1 - \tfrac14\alpha) \epsilon|U||W|.
\end{align*}
On the other hand, since $G' = d_G(X,Y) + c_1 K_{S_1, T_1} + \dots + c_k K_{S_k,T_k}$, we have
\[
e_{G'}(U,W) - d_G(X,Y)|U||W|
= \sum_{i=1}^k c_i |U \cap S_i| |W \cap T_i|.
\]
So
\[
\left|\sum_{i=1}^k c_i |U \cap S_i| |W \cap T_i|\right| \ge (1 - \tfrac14\alpha) \epsilon|U||W|.
\]
Let $u$ and $u_i$ be $|U|$ and $|U \cap S_i|$, each respectively rounded to the nearest integer multiple of $\lfloor \alpha \epsilon^3 |X| /(100k)\rfloor$, for all $1 \le i \le k$. Similarly let $w, w_i$ be $|W|$ and $|W \cap S_i|$, each respectively rounded to the nearest integer multiple of $\lfloor \alpha \epsilon^3 |Y| /(100k)\rfloor$, for all $1 \le i \le k$. The sequence $u, u_1, \dots, u_k, w, w_1, \dots, w_k$ is feasible as witnessed by $\mu = 1_{U \cup W}$. We claim that \eqref{eq:surrogate-reg} holds. Indeed, we have
\[
u \ge |U| - \tfrac{1}{100} \alpha\epsilon^3 |X| \ge (1 - \tfrac{1}{100} \alpha\epsilon^2)\epsilon |X|,
\]
and
\[
w \ge |W| - \tfrac{1}{100} \alpha\epsilon^3 |Y| 
\ge (1 - \tfrac{1}{100} \alpha\epsilon^2)\epsilon |Y|.
\]
Furthermore, we have
\begin{align*}
\left| \sum_{i=1}^k c_i u_i w_i \right|
&\ge \left|\sum_{i=1}^k c_i |U \cap S_i| |W \cap T_i|\right| - \tfrac3{100} \alpha\epsilon^3 |X||Y|
\\
&\ge (1 - \tfrac14\alpha) \epsilon|U||W| - \tfrac3{100} \alpha\epsilon^3 |X||Y|
\\
&\ge (1 - \tfrac14\alpha - \tfrac3{100} \alpha) \epsilon|U||W|
\\
&\ge (1 - \tfrac14 \alpha - \tfrac{3}{100}\alpha) \epsilon ( 1 + \tfrac{1}{100} \alpha\epsilon^2)^{-2} uw
\\
&> (1 - \tfrac12 \alpha) \epsilon uw.
\end{align*}
The first inequality above follows from the fact that for each $i$, \[|u_i-|U \cap S_i|| \le \frac{\alpha \epsilon^3 |X|}{100k},\] \[|u_i-|U \cap S_i|| \le \frac{\alpha \epsilon^3 |Y|}{100k},\] and thus \[\left|u_iw_i-|U \cap S_i||W \cap T_i|\right| \le \frac{3\alpha\epsilon^3 |X||Y|}{100k}.\] The penultimate inequality follows from $u \le |U| + \tfrac1{100} \alpha \epsilon^3 |X| \le (1 + \tfrac1{100} \alpha\epsilon^2)|U|$ and similarly with $w$.
So we have a feasible sequence satisfying \eqref{eq:surrogate-reg}, which is a contradiction.

\medskip

Now suppose instead that the algorithm does find some feasible sequence that satisfies \eqref{eq:surrogate-reg}. By adjusting $\mu$, we may assume that $\mu$ takes $\{0,1\}$-value on all but at most one element in each part in the common refinement partition of $X$ by $S_1, \dots, S_k$, and likewise in $Y$ by $T_1, \dots, T_k$. Let $U \subseteq X$ and $W \subseteq Y$ denote the elements where $\mu$ is positive, we have 
\[
	||U| - u| \le \frac{\alpha\epsilon^3}{100k}|X| + 2^k \le \frac{\alpha\epsilon^3}{50k}|X|
\]
Here the extra $2^k$ term account for rounding up non-integral values of $\mu$. We used the assumption \eqref{eq:k-small} to bound $2^k$.
It thus follows from above, and \eqref{eq:surrogate-reg} that
\begin{align*}
|U| &\ge (1 - \tfrac12 \alpha - \tfrac{1}{50k}\alpha\epsilon^2)\epsilon |X| \ge (1 - \alpha)\epsilon|X|.
\end{align*}
In particular, this means that
\[
	||U| - u| \le \frac{\alpha\epsilon^3}{50k}|X| \le \frac{\alpha\epsilon^2}{(1-\alpha)50k}|U|.
\]
Similarly, we have
\[|W| \ge (1 - \tfrac12 \alpha - \tfrac{1}{50k}\alpha\epsilon^2)\epsilon |Y| \ge (1 - \alpha)\epsilon|Y|,\]
and we have 
\[
	||U \cap S_i| - u_i| \le \frac{\alpha\epsilon^2}{(1-\alpha)50k}|U|,
	\quad \text{for all }1 \le i \le k,
\]
and
\[
	||W| - w|\le \frac{\alpha\epsilon^2}{(1-\alpha)50k}|W|, 
	\quad \text{ and } \quad
	||W \cap T_i| - w_i| 
	\le \frac{\alpha\epsilon^2}{(1-\alpha)50k}|W|,
	\quad \text{for all }1 \le i \le k.
\]
and
\begin{align*}
|d_G(U,W) - d_G(X,Y)|
&\ge |d_{G'}(U,W) - d_G(X,Y)| - |d_{G}(U,W) - d_{G'}(U,W)|
\\
&\ge \frac{1}{|U||W|}\left|\sum_{i=1}^k c_i |U \cap S_i||W \cap T_i| \right| 
- \frac{|X||Y|}{|U||W|}\dd_\square(G,G')
\\
&\ge \frac{1}{|U||W|}\left(\left|\sum_{i=1}^k c_i u_i w_i\right| - \frac{3}{(1-\alpha)50}\alpha\epsilon^2|U||W| )\right) - \frac{1}{4}\alpha\epsilon^3 \frac{|X||Y|}{|U||W|}
\\
&\ge \frac{1}{|U||W|}\left(\left(1-\frac{\alpha}{2}\right)\epsilon u w - \frac{3}{(1-\alpha)50}\alpha\epsilon^2 |U||W| )\right) - \frac{1}{4}\alpha\epsilon
\\&
\ge (1-\alpha)\epsilon.
\end{align*}
Hence the pair $(U,W)$ witnesses that $G$ is not $(1-\alpha)\epsilon$-regular.

\medskip

We will need the following easy corollary of Theorem~\ref{thm:pair} for the next section.

\begin{corollary} \label{cor:check-partition-regular}
There exists an $O_{\epsilon,\alpha,k}(n^2)$ time algorithm, which, given $\epsilon,\alpha,k>0$, a graph $G$ on $n$ vertices, and a partition $\cP$ of the vertex set of $G$ into $k$ parts, does one of the following:
\begin{enumerate}
\item Correctly states that $\cP$ is $(1+\alpha)\epsilon$-regular;
\item Correctly states that $\cP$ is not $\epsilon$-regular.
\end{enumerate}
\end{corollary}

Note that sometimes both options are correct. The algorithm that we give runs in $O(k^2 2^{2^{(\alpha\epsilon)^{-O(1)}}}n^2)$ time.

\begin{proof}
Let $\cP$ be the partition of $V$ into $V_1, \dots, V_k$. Apply the algorithm in Theorem~\ref{thm:pair} to each pair $V_i, V_j$ so that it either correctly states that $(V_i, V_j)$ is $(1+\alpha)\epsilon$-regular or that it is not $\epsilon$-regular. If at least a $(1-\epsilon)$-fraction of pairs are seen to be $(1+\alpha)\epsilon$-regular, then we know that $\cP$ is $(1+\alpha)\epsilon$-regular, otherwise, more than an $\epsilon$-fraction of pairs fail to be $\epsilon$-regular, so that $\cP$ is not $\epsilon$-regular.
\end{proof}

\section{Approximating regularity}

In this section, we prove Theorem~\ref{thm:partition}, reproduced below for convenience.

\begin{theorem*}
There exists an $O_{\epsilon,\alpha, k}(n^2)$ time algorithm, which, given $0<\epsilon,\alpha<1$ and $k$, and a graph $G$ on $n$ vertices that admits an equitable $\epsilon$-regular partition with $k$ parts, outputs an equitable $(1+\alpha)\epsilon$-regular partition of $G$ into $k$ parts.
\end{theorem*}

Here is the algorithm, which runs in $O(2^{2^{(k/(\alpha\epsilon))^{O(1)}}}n^2)$ time. Using Corollary~\ref{cor:reg-sum}, we find $S_1, \dots, S_s,$ $T_1, \dots, T_s \subseteq V$, with $s \le (k/\alpha\epsilon)^{O(1)}$, such that $\dd_\square(G,G') \le \alpha\epsilon/(10k^2)$, where
\[
G' = d(G) + c_1 K_{S_1, T_1} + \dots c_k K_{S_s,T_s}.
\]
Let $\cQ$ denote the partition of $V(G)$ given by the common refinement of the sets $S_1, \dots, S_s, T_1, \dots, T_s$. Let $\cQ$ have $r \le 4^s$ parts, with sizes $q_1, \dots, q_r$. We shall search over all tuples $(q_{i,j})_{1 \le i \le r, 1 \le j \le k}$ of nonnegative integers satisfying all of the following requirements:
\begin{itemize}
	\item $q_i = q_{i,1} + \dots q_{i,k}$ for each $1 \le i \le r$;
	\item each $q_{i,j}$ with $j < k$ is divisible by $\lfloor \alpha\epsilon n/(25rk) \rfloor$ (no divisibility requirements for $q_{i,k}$); and
	\item the sums $\sum_{i=1}^r q_{i,j}$ for different values of $j$ differ from $n/k$ by at most $\alpha\epsilon n/(50k)$.
\end{itemize}
For each eligible tuple $(q_{i,j})$, consider a partition $\cP: V = V_1 \cup \dots \cup V_k$ where $Q_i \cap V_j = q_{i,j}$ (there are many such partitions; pick an arbitrary one). Apply Corollary~\ref{cor:check-partition-regular} to certify that either $\cP$ is $(1+3\alpha/4)\epsilon$-regular or not $(1+\alpha/2)\epsilon$-regular. It turns out that the latter option cannot always be true for all $\cP$ searched, as we assume that $G$ admits some $\epsilon$-regular partition with $k$ parts (we will justify this claim). From this search, we find a $(1+3\alpha/4)\epsilon$-regular partition $\cP$ which is almost equitable in the sense that its parts have sizes differing from $n/k$ by at most $\alpha\epsilon n/(50k)$. We modify $\cP$ by moving a minimum number of vertices to make it equitable. We claim that the resulting partition is $(1+\alpha)\epsilon$-regular.

\medskip

We next analyze the running time of this algorithm. Corollary~\ref{cor:reg-sum} takes  $O(2^{2^{(k/(\alpha\epsilon))^{O(1)}}}n^2)$ time to find the cut norm decomposition. The number of tuples $(q_{i,j})$ is at most $(25rk\alpha^{-1}\epsilon^{-1})^{kr} \le 2^{2^{(k/(\alpha\epsilon))^{O(1)}}}$. For each $(q_{i,j})$, the algorithm in Corollary~\ref{cor:check-partition-regular} takes $O(k^2 2^{2^{(\alpha\epsilon)^{-O(1)}}}n^2)$ time. Therefore, the entire algorithm takes $O(2^{2^{(k/(\alpha\epsilon))^{O(1)}}}n^2)= O_{\alpha,\epsilon,k}(n^2)$ time.

Now we verify correctness. We shall prove the following claims, which together imply the result. Indeed, (1) shows that the algorithm always finds some $(1+3\alpha/4)\epsilon$-regular partition $\cP$, and (2) shows that making $\cP$ equitable by moving a minimum number of vertices between parts results in a $(1+\alpha)\epsilon$-regular partition.
\begin{enumerate}
	\item If a partition $\cP=\{V_1,V_2,...,V_k\}$ of $V$ is $\epsilon$-regular, then we can modify it slightly to obtain $\cP'=\{V_1',V_2',...,V_k'\}$ such that $q_{i,j}=|Q_i \cap V_j'|$ form an eligible tuple, and $\cP'$ is $(1+\alpha/2)\epsilon$-regular for $G$ (so the search would not pass over this $(q_{i,j})$).
	\item If a partition $\cP$ of $V$ is $(1+3\alpha/4)\epsilon$-regular for $G$, then by modifying $\cP$ by adding or deleting at most $\alpha\epsilon n/(50k)$ vertices from each part, the resulting partition is $(1+\alpha)\epsilon$-regular.
\end{enumerate}

In order to show these claims, we first establish a few simple lemmas.

\begin{lemma}\label{prevlem1}
Let $X$, $X'$, $Y$ be vertex subsets of a graph with $X \subset X'$ and $|X| \geq (1-\delta)|X'|$. Then $|d(X',Y)-d(X,Y)| \leq \delta$.  
\end{lemma}
\begin{proof}
We have the identity $$d(X',Y)-d(X,Y)=\frac{e(X' \setminus X,Y)+e(X,Y)}{|X' \setminus X||Y| +|X||Y|}-\frac{e(X,Y)}{|X||Y|}=\left(d(X' \setminus X,Y)-d(X,Y)\right)\frac{|X' \setminus X|}{|X'|}.$$ The lemma follows from noting that densitites are between $0$ and $1$ and $|X' \setminus X| \leq \delta |X'|$.
\end{proof}

Recall that $A \Delta B := (A \setminus B) \cup (B \setminus A)$ denotes the symmetric difference between $A$ and $B$.

\begin{lemma}\label{prevlem2}
If $U$, $U'$, $W$, $W'$ are vertex subsets of a graph with $|U \Delta U'| \leq \delta |U \cup U'|$ and $|W \Delta W'| \leq \delta |W \cup W'|$, then $|d(U,W)-d(U',W')| \leq 2 \delta$. 
\end{lemma}
\begin{proof}
It suffices to prove the lemma in the case $W=W'$ and with the bound $2\delta$ replaced by $\delta$. Indeed, the lemma would then follow by applying this case twice and the triangle inequality. By the triangle inequality and applying Lemma \ref{prevlem1} twice with $X'=U \cup U'$, first with $\delta_1=\frac{|U \cup U'| - |U|}{|U \cup U'|}$ and then with $\delta_2=\frac{|U \cup U'| - |U'|}{|U \cup U'|}$, and finally using $\delta_1+\delta_2 = \frac{|U \Delta U'|}{|U \cup U'|} \leq \delta$, we have 
$$|d(U,W)-d(U',W)|\leq |d(U,W)-d(U \cup U',W)|+|d(U \cup U',W)-d(U',W)| \leq \delta_1+\delta_2 \leq \delta.$$
\end{proof}

\begin{lemma}
Suppose $(V_1,V_2)$ is an $\epsilon$-regular pair of vertex subsets of a graph. Suppose we modify them slightly to $V_1'$ and $V_2'$, with $|V_i \Delta V_i'| \le \delta \epsilon |V_i|$ for $i=1,2$. Then $V_1'$ and $V_2'$ are $\epsilon+4\delta$-regular.
\end{lemma}

\begin{proof}

Clearly we may assume that $\epsilon+4\delta \le 1$. Let $U' \subseteq V'_1$ and $W' \subseteq V'_2$ with $|U'| \ge (\epsilon+4\delta)|V'_1|$ and $|W'| \ge (\epsilon+4\delta)|V'_2|$. Let $U = U' \cap V_1$ and $W = W' \cap V_2$. 
Then we have
\begin{multline*} |U|=|U'|-|U'\setminus V_1| \ge |U'|-|V_1' \setminus V_1| \ge (\epsilon+4\delta)|V_1'|-\delta \epsilon|V_1| \ge \\(\epsilon+4\delta)(|V_1|-\delta \epsilon |V_1|) - \delta \epsilon|V_1|= \epsilon|V_1|+4\delta |V_1|-(1+\epsilon + 4 \delta)\delta\epsilon|V_1| \ge \epsilon|V_1|. \end{multline*}
Similarly $|W| \ge \epsilon |V_2|$. Thus by the regularity of the pair $(V_1,V_2)$, we have 
\[|d(U,W) - d(V_1,V_2)| \le \epsilon.\] 
Now, we have that 
\[|U \Delta U'| \le \delta \epsilon |V_1| \le \delta|U| \le \delta |U \cup U'|,\]
 and similarly $|W \Delta W'| \le \delta |W \cup W'|$, and thus $|d(U,W)-d(U',W')|\le 2\delta$ by Lemma \ref{prevlem2}. Similarly $|d(V'_1,V'_2) - d(V_1,V_2)| \le 2\delta \epsilon \le 2 \delta$. By the triangle inequality, we have $|d(U',W') - d(V'_i,V'_j)|  \le \epsilon + 4\delta$, showing that $(V'_i,V'_j)$ is $(\epsilon + 4\delta)$-regular.
\end{proof}

As a corollary, we have the following:

\begin{corollary}
Let $0<\epsilon,\delta<1$. Let $G$ be a graph with $n$ vertices. Let $\cP$ be a partition of $V(G)$ into $k$ parts, with each part having size at least $n/(2k)$. Suppose that $\cP$ is $\epsilon$-regular for $G$. If we modify $\cP$ by adding or deleting at most $\delta \epsilon |V|/k$ vertices from each part of $\cP$, then the resulting partition is $(\epsilon + 8\delta)$-regular for $G$.
\end{corollary}

\begin{proof}
Indeed, for any part $V_i$ of $\cP$, if we let $V_i'$ be its modification, then $|V_i \Delta V_i'| \le \delta \epsilon |V|/k \le 2 \delta \epsilon |V_i|$. This means that if a pair $(V_i,V_j)$ was $\epsilon$-regular, then after the modification it is  $(\epsilon+8\delta)$-regular, and so the proportion of pairs that are not $\epsilon+8\delta$-regular is at most $\epsilon \le \epsilon+8\delta$.
\end{proof}

Now we prove claim (1) above. Let $\cP$ be an equitable $\epsilon$-regular partition of $G$. Since $\dd_\square(G,G') \le \alpha\epsilon/(10k^2)$, $\cP$ is $(1 + \alpha/10)\epsilon$-regular for $G'$. In $G'$, edges between the same parts of $\cQ$ have equal weights, and we can take $\cP'$ such that $|Q_i \cap V_j|$ differs from $|Q_i \cap V'_j|$ by at most $\alpha\epsilon n/(50rk)$ for each $i,j$. This means that $\cP'$ can be taken so that $V_j$ and $V_j'$ differ by at most $\alpha \epsilon n/(50k)$ for each $j$, so it follows from the lemma above that $\cP'$ is $(1 + 3\alpha/10)\epsilon$-regular for $G'$. Therefore, $\cP'$ must be $(1+\alpha/2)\epsilon$-regular for $G$. 

The claim (2) follows immediately from the corollary above.

\section{Permutation regularity lemma}
\label{sec:perm-reg} 

In this section we give a new proof of a regularity lemma for permutations that requires fewer parts than previous results in literature. To define regular partitions for permutations, it is natural to state it as a special case in a more general setting for matrices.

Let $Y = (y_{ij})$ be a $n \times n$ matrix. We use \emph{interval} to
mean a subset of $[n]$ of consecutive
integers. For any intervals $I, J$ of $[n]$, we write
\[
d_Y(I, J) := \frac{1}{|I||J|} \sum_{i \in I, j \in J} y_{ij}.
\]

\begin{definition}
  Let $Y$ be a $n \times n$ square matrix. Let $I, J
  \subseteq [n]$ be intervals. We say that $(I, J)$ is \emph{interval
    $\epsilon$-regular} for $Y$ if for all subintervals $A \subseteq I$
  and $B \subseteq J$ with $|A| \geq \epsilon |I|$ and $|B| \ge \epsilon
  |J|$ one has
  \[
  |d_Y(A, B) - d_Y(I, J)| \le \epsilon.
  \]
  Let $\cP$ be a partition of $[n]$ into $k$ intervals. We say that
  $\cP$ is interval \emph{$\epsilon$-regular} for $Y$ if all except at most
  $\epsilon k^2$ pairs of intervals $(I, J)$ of $\cP$ are
  interval $\epsilon$-regular for $Y$.
\end{definition}

\begin{definition}
  We say that $\cP$ is an \emph{equipartition} of $[n]$ if every
  pair of parts in $\cP$
  differ in size by at most one.
\end{definition}

Here is the regularity lemma for interval regular partitions.

\begin{theorem}[Interval regularity lemma] \label{thm:int-reg-matrix} For every $\epsilon > 0$
  and positive integer $m$ there is some $M = m^{O(1)}
  \epsilon^{-O(\epsilon^{-5})} $ with the
  following property. For every $n \in \NN$, and $n \times n$ matrix
  $Y = (y_{ij})$ with $[0,1]$-valued entries, there is some integer $k
  \in [m, M]$ so that every equipartition of $[n]$ into $k$
  intervals is interval $\epsilon$-regular for $Y$.
\end{theorem}

\begin{remark}
If $n \le M$, then we can take the partition of $[n]$ into singletons. Otherwise, our proof will show that one can pick $k$ from a small set of choices: one can take $k = mq^i$, where $q = \lceil 16\epsilon^{-3} \rceil$ and $0 \le i < \lceil 4\epsilon^{-5} \rceil$ is some integer.
\end{remark}

Theorem~\ref{thm:int-reg-matrix} has the following immediate
consequence for permutation regularity. Given a permutation $\sigma \colon [n] \to [n]$, associate to it the
$n\times n$
matrix $Y^{\sigma}$ defined by
\[
y_{ij} = \begin{cases} 1 & \text{if } \sigma(i) < j \\
  0 & \text{otherwise}.
\end{cases}
\]
A partition of $[n]$ into
  intervals is said
  to be \emph{$\epsilon$-regular} for $\sigma$ if it is interval $\epsilon$-regular for
  the associated matrix $Y^\sigma$.

\begin{theorem}[Permutation regularity lemma]
  \label{thm:perm-reg-lemma}
  For every $\epsilon > 0$ and positive integer $m$, there exist
  $M = m^{O(1)} \epsilon^{-O(\epsilon^{-5})}$ with the following property. Let $n \neq n_0$ and $\sigma$
  be a permutation of $[n]$. Then for some integer $k \in [m, M]$,
  every equitable partition of $[n]$ into $k$ intervals is
  $\epsilon$-regular for $\sigma$.
  \qed
\end{theorem}

An early form of this permutation regularity lemma was first proved by
Cooper~\cite{Coo06}. The above form was proved in \cite{HKS12} with $M$ being a
tower exponential of height $O(\epsilon^{-5})$. Our version requires a much smaller $M$.

\subsection{Interval regular partitions for functions}

We first prove the interval regularity lemma for functions. It is
somewhat cleaner to work with partitions of the real interval $[0,1]$ into
exactly equal-length subintervals, instead of equitable partitions of $[n]$. The measure theoretic approach has the slight advantage that it allows us to defer divisibility issues of $n$ until the end.

Let $f \colon [0,1]^2 \to [0,1]$ be a measurable function. For any
intervals $I, J \subseteq[0,1]$ we write
\[
d_f(I, J) := \frac{1}{\lambda(I)\lambda(J)} \int_{I \times J} f(x,y)
\, dxdy.
\]
Here $\lambda$ denotes the Lebesgue measure.

\begin{definition}
  Let $f \colon [0,1]^2 \to [0,1]$ be a measurable function. Let $I, J
  \subseteq [0,1]$ be intervals. We say that $(I, J)$ is
  \emph{interval $\epsilon$-regular} for $f$ if for all subintervals $A
  \subseteq I$ and $B \subseteq J$ with $\lambda(A) \ge \epsilon
  \lambda(I)$ and $\lambda(B) \ge \epsilon \lambda(J)$ one has
  \[
  |d_f(A, B) - d_f(I, J)| \le \epsilon.
  \]

  Let $\cP$ a partition of $[0,1]$ into $k$ intervals. We say that
  $\cP$ is \emph{interval $\epsilon$-regular} for $f$ if all except at most $\epsilon k^2$
  pairs of intervals $(I, J)$ of $\cP$ are interval
  $\epsilon$-regular for $f$.
\end{definition}

\begin{theorem} \label{thm:int-reg-function} For every $\epsilon > 0$
  and positive integer $m$ there is some $M =
  m\epsilon^{-O(\epsilon^{-5})}$ with the following property. For
  every measurable function $f \colon [0,1]^2 \to [0,1]$, there is
  some integer $k \in [m, M]$ such that the partition of $[0,1]$
  into $k$ equal-length intervals  $[0, 1/k) \cup [1/k, 2/k) \cup \dots \cup [(k-1)/k,1]$ is interval $\epsilon$-regular for $f$.
\end{theorem}

\begin{remark}
In Theorem~\ref{thm:int-reg-function}, it is possible to take $k = mq^i$, where $q = \lceil 16\epsilon^{-3} \rceil$ and $0 \le i < \lceil 4\epsilon^{-5} \rceil$ is some integer.
\end{remark}

Before proving Theorem~\ref{thm:int-reg-function}, we first prove a
lemma showing that the density $d_f(A, B)$ does not change very much
if $A$ and $B$ are changed only slightly. 

\begin{lemma} \label{lem:density-diff}
  Let $f \colon [0,1]^2 \to [0,1]$ be a measurable function. For any
  intervals $A,
  A', B, B' \subseteq [0,1]$ we have
  \[
  |d_f(A, B) - d_f(A', B')| \leq \frac{2 \lambda((A \times B) \Delta
    (A' \times B'))}{\lambda(A)\lambda(B)}.
  \]
\end{lemma}

\begin{proof}
By the triangle inequality,
\begin{align*}
&\lambda(A)\lambda(B) |d_f(A,B) - d_f(A',B')|
\\
&\le
\left| \lambda(A) \lambda(B) d_f(A, B) - \lambda(A') \lambda(B') d_f(A',
  B') \right| + d_f(A',B') |\lambda(A)\lambda(B) -
\lambda(A')\lambda(B')|
\\&\le \left|\int_{A \times B} f\,d\lambda - \int_{A' \times B'}
  f\,d\lambda\right| +   |\lambda(A)\lambda(B) -
\lambda(A')\lambda(B')|
\\&\le 2\lambda((A \times B) \Delta (A' \times B')).
\end{align*}
\end{proof}

The bound in Lemma \ref{lem:density-diff} can be improved by a factor $2$ by following the proof of Lemma \ref{prevlem2}. 

\begin{proof}[Proof of Theorem~\ref{thm:int-reg-function}]
  Let $f_k$ denote the function obtained from $f$ by replacing its
  value inside each box $[i/k, (i+1)/k) \times [j/k, (j+1)/k)$ by its
  average inside that box, i.e.,
  \[
  f_k(x,y) := k^2 \int_{\left[ \frac{i}{k}, \frac{i+1}{k} \right) \times \left[ \frac{j}{k}, \frac{j+1}{k} \right)} f \,d\lambda \qquad \text{if
  } (x,y) \in \left[ \frac{i}{k}, \frac{i+1}{k} \right) \times \left[ \frac{j}{k}, \frac{j+1}{k} \right)
  \]
  for $i, j = 0,1, \dots, k-1$ (when $i$ or $j$ equals $k-1$, the
  corresponding interval should be modified to be closed on the
  right). Write
  \[
  \|f\|_2 := \left(\int_{[0,1]^2} |f|^2 \,d\lambda\right)^{1/2}
  \]
  for the $L^2$ norm.
  
  Let $q = \lceil 16 \epsilon^{-3} \rceil$. Consider the sequence 
  $
  f_m, f_{mq}, f_{mq^2}, \dots$.
  Since $0 \leq \|f_k\|_2 \leq 1$
  for all $k$, there exists some $k = mq^i$ for $0 \le i < \ceil{4\epsilon^{-5}}$ such that
  \begin{equation} \label{eq:f_kq-f_k}
    \|f_{kq}\|_2^2 \le \|f_k\|_2^2 + \frac{\epsilon^5}{4}.
  \end{equation}
  We will show that the partition of $[0,1]$ into $k$ equal-length intervals  
  is interval $\epsilon$-regular. Indeed, if this were not the case, then there
  exists more than $\epsilon k^2$ irregular pairs of intervals $(I,
  J)$, where $I = [i/k, (i+1)/k)$ and $J = [j/k, (j+1)/k)$ for some
  integers $i$ and $j$. Due
  to the irregularity, there exist subintervals $A \subseteq I$ and $B
  \subseteq J$ such that $\lambda(A) \ge \epsilon \lambda(I)$,
  $\lambda(B) \ge \epsilon \lambda(J)$, and
  \begin{equation} \label{eq:int-reg-func-discrep}
  |d_f(I, J) - d_f(A, B)| > \epsilon.
\end{equation}
Let $A'$ be the smallest interval containing $A$ with both ends
  being multiples of $1/(kq)$. Note that $A' \subseteq I$. Similarly
  define $B'$. We see that $A' \times B'$ contains $A \times B$, and
  the difference in area is at most $4/(k^2q)$. By
  Lemma~\ref{lem:density-diff},
  \[
  |d_f(A, B) - d_f(A', B')| \leq \frac{2 (4/(k^2q))}{(\epsilon/k)^2} =
  \frac{8}{q\epsilon^2} = \frac{8}{\lceil 16 \epsilon^{-3}\rceil \epsilon^2} \le \frac{\epsilon}{2}.
  \]
By \eqref{eq:int-reg-func-discrep} we have
\[
|d_f(I, J) - d_f(A', B')| > \frac{\epsilon}{2}.
\]
Since the endpoints of $I$ and $J$ are multiples of $1/k$ and those of $A'$
and $B'$ are multiples of $1/(kq)$, the function $f_k - f_{kq}$ has average value $d_f(I, J) - d_f(A',
B')$ over the box $A' \times B'$. So the contribution to $\|f_k - f_{kq}\|_2^2$ from $A' \times
B'$ is at least $\lambda(A') \lambda(B') (\epsilon/2)^2 \geq
\epsilon^4/(4k^2)$. As there are more than $\epsilon k^2$ irregular
pairs $(I, J)$, and all the rectangles  $I \times J$ are disjoint, we have
\[
\|f_k - f_{kq}\|_2^2 > \frac{\epsilon^5}{4}.
\]
Note that
\[
\int_{[0,1]^2} (f_k - f_{kq}) f_k  \, d\lambda = 0
\]
since $f_k$ is constant over each box $[i/k, (i+1)/k) \times [j/k,
(j+1)/k)$, and $f_{kq}$ averages to $f_k$ on this box. Thus $f_k$ and
$f_k - f_{kq}$ are orthogonal, so by the Pythagorean theorem,
\[
\|f_{kq}\|_2^2 = \|f_{k} - (f_k - f_{kq})\|_2^2 = \|f_k\|_2^2 + \|f_k
- f_{kq}\|_2^2 > \|f_k\|_2^2 + \frac{\epsilon^5}{4},
\]
which contradicts \eqref{eq:f_kq-f_k}. It follows that the partition
of $[0,1]$ into $k$ equal-length intervals is interval $\epsilon$-regular for $f$.
\end{proof}

\subsection{Dealing with equitable partitions}

Here is a lemma that will be useful for the proof of
Theorem~\ref{thm:int-reg-matrix}. It says that $(I, J)$ being interval
regular is robust under changing $I$ and $J$ by a small amount.

\begin{lemma} \label{thm:int-eps-reg-nudge}
  Let $f \colon [0,1]^2 \to [0,1]$ be a measurable function. Let $I,
  I', J, J' \subseteq [0,1]$. Let $0 < \epsilon \leq 1$. Let $\epsilon' > 0$ be
  a quantity less than each of
  \[
  \epsilon - \frac{4 \lambda( (I \times J)\Delta (I' \times
    J'))}{\epsilon^2 \lambda(I)\lambda(J)}, \quad \frac{\lambda(I)
    \epsilon - \lambda(I \setminus I')}{\lambda(I')}, \quad \frac{\lambda(J)
    \epsilon - \lambda(J \setminus J')}{\lambda(J')}.
  \]
  If $(I', J')$ is interval $\epsilon'$-regular for $f$, then $(I, J)$
  is interval $\epsilon$-regular for $f$.
\end{lemma}

\begin{proof}
  Let $A \subseteq I$ and $B \subseteq J$ be subintervals such that
  $\lambda(A) \ge \epsilon \lambda(I)$ and $\lambda(B) \ge \epsilon
  \lambda(J)$. Let $A' = A \cap I'$ and $B' = B \cap J'$. The second
  and third hypotheses about $\epsilon'$ above imply that $\lambda(A')
  \ge \epsilon' \lambda(I')$ and $\lambda(B')
  \ge \epsilon' \lambda(J')$. Since $(I', J')$ is $\epsilon'$-regular
  for $f$, we have
  \[
  |d_f(A', B') - d_f(I', J')| \leq \epsilon'.
  \]
  By Lemma~\ref{lem:density-diff}, we have
  \[
  |d_f(I, J) - d_f(I', J')| \le \frac{2 \lambda( (I \times J)\Delta (I' \times
    J'))}{\lambda(I)\lambda(J)}
  \]
  and
  \[
  |d_f(A, B) - d_f(A', B')| \le \frac{2 \lambda( (A \times B)\Delta (A' \times
    B'))}{\lambda(A)\lambda(B)}
  \le \frac{2 \lambda( (I \times J)\Delta (I' \times
    J'))}{\epsilon^2 \lambda(I)\lambda(J)}.
  \]
  It follows by the triangle inequality and the first hypotheses on
  $\epsilon'$ that
  \[
  |d_f(A, B) - d_f(I, J)| \leq \epsilon,
  \]
  which proves that $(I,J)$ is $\epsilon$ regular for $f$.
\end{proof}

\begin{proof}[Proof of Theorem~\ref{thm:int-reg-matrix}]
  Let $f \colon [0,1]^2 \to [0,1]$ be the function that takes constant value $y_{ij}$ on the rectangle $[(i-1)/n,
  i/n) \times [(j-1)/n, j/n)$, for each $1 \le i, j \le n$. By
  Theorem~\ref{thm:int-reg-function}, there is some $k \in [m, m\epsilon^{-O(\epsilon^{-5})}]$ so that the
  partition of $[0,1]$ into $k$ equal-length intervals
  interval $(\epsilon/2)$-regular.
  
  Any equitable partition $\cP$ of $[n]$ into sets of sizes $c_1, \dots, c_k$, gives rise to a partition $\cQ$ of $[0,1]$ into intervals of length $c_1/n, \dots, c_k/n$.   Since $\cP$ is an equitable partition, the $i$-th interval $I_i$ of
  $\cQ$ differs, in terms of symmetric
  difference, from $[(i-1)/k, i/k)$  by at most $k/n$ in measure. It follows from
  Lemma~\ref{thm:int-eps-reg-nudge} that if $n$ is large enough, say, $n \ge 100k^3 \epsilon^{-3}$, then $(I_i, J_i)$ is interval
  $\epsilon$-regular for $f$ whenever $[(i-1)/k, i/k) \times [(j-1)/k,
  j/k)$ is interval $(\epsilon/2)$-regular for $f$. It follows that $\cQ$ is interval $\epsilon$-regular for $f$. 
  
  When $n < 100k^3\epsilon^{-3}$, we can take the partition of $[n]$ into singletons, which is trivially interval $\epsilon$-regular.
\end{proof}

\end{document}



\maketitle

Our paper \cite{FLZ17e} contained two errors.

We retract Corollary 3.5.
We claimed that it follows from a recent algorithmic version \cite{DKMRS12e,DKMRS15e} of the Frieze--Kannan regularity lemma (see Theorems 3.2 and Theorem 3.3 in \cite{FLZ17e}). Unfortunately there is an error in our application of these algorithmic results. Although the algorithms \cite{DKMRS12e,DKMRS15e} can test whether a \emph{partition} is weakly regular, it is unclear how to apply it to an intermediate step in our proposed algorithm. Specifically, we would need an algorithm that, given
\[
G'=d(G) + c_1K_{S_1, T_1} + \dots + c_l K_{S_l,T_l},
\] 
either (1) correctly states that $\dd_\square(G,G') \le \epsilon$, or (2) outputs sets $S$ and $T$ such that $|e_G(S,T)-e_{G'}(S,T)|>\epsilon^{-O(1)}$.
In particular, the fact that $G'$ may not have bounded weights presents a challenge in applying results from \cite{DKMRS12e,DKMRS15e}.

A second error is in the proof of Theorem 1.4 in \cite{FLZ17e}, which we also retract. We erred in applying a counting lemma \cite[Lemma 4.1]{FLZ17e} to claim that if $G'$, as above, approximates $G$ in cut norm, then $G$ and $G'$ have similar $H$-densities. However, the counting lemma requires the edge-weights to be bounded by $1$, which is not necessarily true in this case. 

\medskip

We have been able to fix Corollary 3.5. In our new paper~\cite{FLZ17new}, we prove a strengthened version of the algorithm whose running time has improved dependence on $\epsilon$ and $n$. Specifically, we prove the following theorem to replace Corollary 3.5.
\begin{theorem} \label{thmalgweakregbest}
There is a deterministic algorithm that, given $\epsilon>0$ and an $n$-vertex graph $G$, outputs, in $\epsilon^{-O(1)}n^2$ time, subsets $S_1,S_2,...,S_r,T_1,T_2,...,T_r \subseteq V(G)$ and $c_1,c_2,...,c_r \in \{-\frac{\epsilon^8}{300},\frac{\epsilon^8}{300}\}$ for some $r = O(\epsilon^{-16})$, such that
	\[
	\dd_\square(G, d(G) + c_1K_{S_1, T_1} + \dots + c_r K_{S_r,T_r}) \le \epsilon.
	\] 
\end{theorem}

We have been able to salvage the following weaker result to replace Theorem 1.4 in \cite{FLZ17e} (with an improved dependence on $n$ thanks to Theorem \ref{thmalgweakregbest} above). 

\begin{theorem}\label{thm:countingalgthmgood}
There is a deterministic algorithm that, given $\epsilon > 0$, a graph $H$, and an $n$-vertex graph $G$, outputs, in $O(\epsilon^{-O_H(1)} n^2)$ time, the number of copies of $H$ in $G$ up to an additive error of at most $\epsilon n^{v(H)}$.
\end{theorem}

See \cite{FLZ17new} for proofs and discussions.
